\documentclass[12pt,reqno]{amsart}

\usepackage{amsmath,amssymb,amsfonts,amsthm,latexsym,graphicx,multirow,enumerate,hyperref,todonotes}

\newcommand{\Aut}{\mathrm{Aut}}

\newcommand{\calC}{\mathcal{C}}

\newtheorem{theorem}{Theorem}[section]
\newtheorem{lemma}[theorem]{Lemma}
\newtheorem{proposition}[theorem]{Proposition}

\newtheorem{problem}[theorem]{Problem}
\newtheorem{conjecture}[theorem]{Conjecture}
\newtheorem{question}[theorem]{Question}
\theoremstyle{definition}
\newtheorem{definition}[theorem]{Definition}

\long\def\delete#1{}

\usepackage{color}

\definecolor{Blue}{rgb}{0,0,1}
\definecolor{Red}{rgb}{1,0,0}
\definecolor{DarkGreen}{rgb}{0,0.6,0}
\definecolor{DarkYellow}{rgb}{1,1,0.2}
\definecolor{DarkPurple}{rgb}{.6,0,1}

\usepackage{xcolor}
\usepackage[normalem]{ulem}

\begin{document}
\openup 0.5\jot

\title[Stability of graph pairs]{Stability of graph pairs involving vertex-transitive graphs}

\author[Qin]{Yan-Li Qin}
\address{School of Statistics\\Capital University of Economics and Business\\Beijing, 100070\\P.R. China}
\email{ylqin@cueb.edu.cn}

\author[Xia]{Binzhou Xia}
\address{School of Mathematics and Statistics\\The University of Melbourne\\Parkville, VIC 3010\\Australia}
\email{binzhoux@unimelb.edu.au}

\author[Zhou]{Sanming Zhou}
\address{School of Mathematics and Statistics\\The University of Melbourne\\Parkville, VIC 3010\\Australia}
\email{sanming@unimelb.edu.au}

\begin{abstract}
A pair of graphs $(\Gamma,\Sigma)$ is said to be stable if the full automorphism group of $\Gamma\times\Sigma$ is isomorphic to the product of the full automorphism groups of $\Gamma$ and $\Sigma$ and unstable otherwise, where $\Gamma\times\Sigma$ is the direct product of $\Gamma$ and $\Sigma$. 
% An unstable graph pair $(\Gamma,\Sigma)$ is said to be nontrivially unstable if $\Gamma$ and $\Sigma$ are connected coprime graphs, at least one of them is non-bipartite, and each of them has the property that different vertices have distinct neighborhoods. 
In this paper, we reduce the study of the stability of any pair of regular graphs $(\Gamma,\Sigma)$ with coprime valencies and vertex-transitive $\Sigma$ to that of $(\Gamma,K_2)$. Since the latter is well studied in the literature, this enables us to determine the stability of any pair of regular graphs $(\Gamma,\Sigma)$ with coprime valencies in the case when $\Sigma$ is vertex-transitve and the stability of $(\Gamma,K_2)$ is known. 

\end{abstract}

\maketitle

%==================================================================================================================================================================================================
\section{Introduction}
%==================================================================================================================================================================================================

We only consider finite undirected graphs with no loops or parallel edges. The vertex set, edge set and full automorphism group of a graph $\Gamma$ are denoted by $V(\Gamma)$, $E(\Gamma)$ and $\Aut(\Gamma)$, respectively, and the cardinality of $V(\Gamma)$ is referred to as the order of $\Gamma$. The edge between two adjacent vertices $u, v$ is written as $\{u, v\}$. The complete graph of order $m$ is denoted by $K_m$, and the cycle of order $m$ is denoted by $C_m$. A graph is \emph{vertex-transitive} if its automorphism group is transitive on its vertex set, and a graph is \emph{arc-transitive} if its automorphism group is transitive on its arc set, where an \emph{arc} is an ordered pair of adjacent vertices.

Let $\Gamma$ and $\Sigma$ be graphs. The \emph{direct product} \cite{HIK2011} of $\Gamma$ and $\Sigma$, denoted by $\Gamma\times\Sigma$, is the graph with vertex set $V(\Gamma)\times V(\Sigma)$ and edge set $\{\{(u,x),(v,y)\}: \{u, v\} \in E(\Gamma) \text{ and } \{x, y\} \in E(\Sigma)\}$. It follows from this definition that the direct product $\Aut(\Gamma)\times\Aut(\Sigma)$ of $\Aut(\Gamma)$ and $\Aut(\Sigma)$ is isomorphic to a subgroup of $\Aut(\Gamma\times\Sigma)$.

\begin{definition}
\label{def:st}
A graph pair $(\Gamma,\Sigma)$ is said to be \emph{stable} if $\Aut(\Gamma)\times\Aut(\Sigma)$ is isomorphic to $\Aut(\Gamma\times\Sigma)$ and \emph{unstable} otherwise. 
\end{definition}

This concept in its general form was introduced in \cite[Definition 1.1]{QXZZ2021} as a generalization of the notion of the stability of graphs introduced in \cite{MSZ1989}. In fact, a graph $\Gamma$ is stable in terms of \cite{MSZ1989} if and only if the graph pair $(\Gamma,K_2)$ is stable in terms of Definition \ref{def:st}. 

The stability of graphs is closely related to symmetric $(0, 1)$ matrices \cite{MSZ1989}, regular embeddings of canonical double covers~\cite{NS1996}, two-fold automorphisms of graphs~\cite{LMS2015}, and generalized Cayley graphs~\cite{MSZ1992}. As such it has received considerable attention especially in recent years (see, for example, \cite{FH2022,HMM2021,M2021,QXZ2019,Surowski2001,Surowski2003,Wilson2008}). In~\cite{Wilson2008}, Wilson gave four sufficient conditions for a graph to be unstable and used these results to study the instability of circulant graphs, generalized Petersen graphs, and three other families of graphs. In \cite{QXZ2019}, the authors answered an open question in \cite{Wilson2008} about the stability of arc-transitive circulant graphs and in the meantime constructed an infinite family of counterexamples to a conjecture in~\cite{MSZ1989}. Subsequently, a conjecture of the authors in~\cite{QXZ2019} about unstable circulant graphs of odd order was proved by Fernandez and Hujdurovi\'{c} in~\cite{FH2022}, and in turn an open question in~\cite{FH2022} about the stability of Cayley graphs on abelian groups of odd order was recently answered by Morris~\cite{M2021}. A conjecture in~\cite{Wilson2008} about the stability of generalized Petersen graphs was proved by the authors in~\cite{QXZ2021}. 

Compared with the stability of graphs, there are only few known results on the stability of general graph pairs in the literature. It turns out that the stability of graph pairs is closely related to the following concepts: A graph is called \emph{$R$-thin}~\cite{HIK2011} or \emph{vertex-determining}~\cite{QXZ2019, Wilson2008} if no two vertices have the same neighborhood in the graph, and graphs that are not $R$-thin are said to be \emph{$R$-thick}~\cite{HIK2011}. Two graphs $\Gamma$ and $\Sigma$ are said to be \emph{coprime} if there is no graph $\Delta$ of order greater than $1$ such that $\Gamma\cong\Gamma_1\times\Delta$ and $\Sigma\cong\Sigma_1\times\Delta$ for some graphs $\Gamma_1$ and $\Sigma_1$. 
The following result, which is a special case of \cite[Theorem~8.18]{HIK2011}, gives sufficient conditions for a graph pair to be stable. (For the condition of $\Gamma\times\Sigma$ being connected, non-bipartite and $R$-thin, respectively, in terms of $\Gamma$ and $\Sigma$, the reader is referred to Lemma~\ref{cvn}.)

\begin{theorem}\label{classic} 
Let $\Gamma$ and $\Sigma$ be coprime graphs. If $\Gamma\times\Sigma$ is connected, non-bipartite and $R$-thin, then $(\Gamma, \Sigma)$ is stable.
\end{theorem}

%Based on Theorem~\ref{classic}, is natural to study the stability of a graph pair $(\Gamma,\Sigma)$ with $\Gamma\times\Sigma$ disconnected, or bipartite, or $R$-thick. In fact, we can obtain that $(\Gamma,\Sigma)$ is unstable if $\Gamma\times\Sigma$ is either $R$-thick, or disconnected with $\Aut(\Gamma)\neq1$ and $\Aut(\Sigma)\neq1$ by Lemma~\ref{cvn} and Lemma~\ref{TriviallyStable}.??? 

On the other hand, the following result from \cite[Theorem~1.3]{QXZZ2021} gives necessary conditions for a graph pair to be stable.

\begin{theorem}\label{TriviallyStable}
Let $(\Gamma,\Sigma)$ be a stable pair of graphs. Then $\Gamma$ and $\Sigma$ are coprime $R$-thin graphs. Moreover, if in addition both $\Aut(\Gamma)$ and $\Aut(\Sigma)$ are nontrivial groups, then both $\Gamma$ and $\Sigma$ are connected and at least one of them is non-bipartite.
\end{theorem}

In studying the stability of graph pairs it is rather natural to investigate the case when the graphs involved have nontrivial automorphism groups. Under this assumption Theorem~\ref{TriviallyStable} implies that we can focus on those pairs $(\Gamma,\Sigma)$ such that $\Gamma$ and $\Sigma$ are connected coprime $R$-thin graphs and at least one of them is non-bipartite. This motivated the following definition which was introduced in \cite[Definition 1.4]{QXZZ2021}.

\begin{definition}
\label{def-nontrivial-stable}
An unstable graph pair $(\Gamma,\Sigma)$ is called \emph{nontrivially unstable} if $\Gamma$ and $\Sigma$ are connected coprime $R$-thin graphs and at least one of them is non-bipartite, and \emph{trivially unstable} otherwise. 
\end{definition}

This definition is consistent with its counterpart for graphs: A graph $\Gamma$ is nontrivially unstable or trivially unstable in terms of \cite{Wilson2008} if and only if $(\Gamma,K_2)$ is nontrivially unstable or trivially unstable, respectively.

%\begin{hypothesis}\label{hy}
%Let $(\Gamma,\Sigma)$ be a graph pair such that both graphs are connected and regular with coprime valencies and one of them is bipartite, the other is non-bipartite, and $\Sigma$ is vertex-transitive. 
%\end{hypothesis}

% Needless to say, the following problem is of central importance to the study of the stability of graph pairs. 

%Since a graph not coprime to $K_2$ is necessarily bipartite, we see that $(\Gamma, K_2)$ is nontrivially unstable if and only if $\Gamma$ is a non-bipartite connected $R$-thin unstable graph. Such a graph $\Gamma$ is called \emph{nontrivially unstable} by Wilson in~\cite{Wilson2008} in a study of the stability of graphs. So the above definition of nontrivially unstable graph pairs generalizes the concept of nontrivially unstable graphs. 

In general, it is not easy to test by definition whether two graphs are coprime. However, there are some sufficient conditions for two graphs to be coprime that are easy to check. For example, any two graphs of coprime orders must be coprime, and any two regular graphs of coprime valencies with at least one non-bipartite must be coprime. 
Thus, it is natural to study the following problem.

\begin{problem}\label{prob}
Determine the nontrivial instability of a graph pair $(\Gamma,\Sigma)$ such that $\Gamma$ and $\Sigma$ are regular graphs of coprime valencies.
\end{problem}

The study of this problem in its general form was initiated by the authors in \cite{QXZZ2021}. Among other things we gave in \cite[Theorem 1.8]{QXZZ2021} a characterization of nontrivial unstable graph pairs $(\Gamma, \Sigma)$ in the case when both $\Gamma$ and $\Sigma$ are connected, $R$-thin and regular with coprime valencies and $\Sigma$ is vertex-transitive (see Lemma~\ref{Coprime} in the next section). The well-studied problem of determining the nontrivial instability of a regular graph is a special case of Problem~\ref{prob}, because $\Gamma$ and $K_2$ are of coprime valencies for any regular graph $\Gamma$.  

In this paper we continue our study of Problem~\ref{prob} with a focus on establishing connections between the nontrivial instability of a graph pair $(\Gamma,\Sigma)$ and that of the graph $\Gamma$. Our main result stated below achieves this goal for all vertex-transitive graphs $\Sigma$.

\begin{theorem}\label{thm1}
Let $\Gamma$ and $\Sigma$ be regular graphs of coprime valencies with $\Sigma$ vertex-transitive.  
\begin{enumerate}[{\rm (a)}]
\item Suppose that $\Sigma$ is connected, $R$-thin and bipartite. Then $(\Gamma,\Sigma)$ is nontrivially unstable if and only if $\Gamma$ is nontrivially unstable. 
\item Suppose that $\Sigma$ is disconnected, or $R$-thick, or non-bipartite. Then $(\Gamma,\Sigma)$ cannot be nontrivially unstable.
\end{enumerate}
\end{theorem}

% A counterpart of Theorem~\ref{thm-tour} is the next main result of this paper, which deals with the case when $\Sigma$ is non-bipartite.

% \begin{theorem}\label{thm1}
% There is no nontrivially unstable graph pair $(\Gamma,\Sigma)$ such that $\Gamma$ and $\Sigma$ are regular graphs of coprime valencies with $\Sigma$ vertex-transitive and non-bipartite. 
% \end{theorem}

This result reduces the quest for nontrivial instability of $(\Gamma,\Sigma)$ to that for nontrivial instability of $\Gamma$ in the case when $\Gamma$ and $\Sigma$ are regular of coprime valencies and $\Sigma$ is vertex-transitive. This enables us to obtain many nontrivially unstable graph pairs from known nontrivially unstable graphs. For example, in \cite{Wilson2008} Wilson gave four sufficient conditions for a circulant graph to be nontrivially unstable (see \cite{QXZ2019} for an amendment to one of these conditions). By Theorem \ref{thm1}, any such nontrivially unstable circulant $\Gamma$ and any connected, vertex-transitive, $R$-thin and bipartite graph $\Sigma$ with valency coprime to the valency of $\Gamma$ give rise to a nontrivially unstable graph pair $(\Gamma,\Sigma)$. 
% \zhou{please check whether this paragraph makes sense.}

The rest of this paper is organized as follows. In the next section we will set up notation and recall a few known results on stability of graphs and stability of graph pairs. The proof of Theorem~\ref{thm1} will be given in Section~\ref{sec-proof1}. As applications of Theorem~\ref{thm1}, we will discuss the stability of $(\Gamma,K_m)$ and $(\Gamma,C_m)$ in Section~\ref{sec-application}.  

%==================================================================================================================================================================================================
\section{Preliminaries}
%==================================================================================================================================================================================================

For a graph $\Gamma$ and a vertex $u$ of $\Gamma$, the \emph{neighborhood} of $u$ in $\Gamma$, denoted by $N_\Gamma(u)$, is the set of vertices adjacent to $u$ in $\Gamma$. For two adjacent vertices $u, v$ in a graph, the edge between them is denoted by the unordered pair $\{u, v\}$. 

% Let $\Gamma$ be a graph. A \emph{walk} in $\Gamma$ is a finite non-null sequence $W=v_0\ e_1\ v_1\ e_2\ v_2\ \ldots\ e_\ell\ v_\ell$, whose terms are alternately vertices and edges such that for $1\leq i\leq\ell$, the ends of $e_i$ are $v_{i-1}$ and $v_i$. If the edges $e_1, e_2, \ldots, e_\ell$ of a walk $W$ are distinct, then $W$ is called a \emph{trail}. A trail that traverses every edge of $\Gamma$ is called an \emph{Euler trail} of $\Gamma$. A \emph{tour} of $\Gamma$ is a closed walk that traverses each edge of $\Gamma$ at least once. An \emph{Euler tour} is a tour which traverses each edge exactly once (in other words, a closed Euler trail). A graph is \emph{eulerian} if it contains an Euler tour. 

% The following result is well known, see~\cite[Theorem~4.1]{BM1976} for example. 

% \begin{lemma}\label{Euler-cycle}
% A nonempty connected graph is eulerian if and only if it has no vertices of odd degree.
% \end{lemma}

%-------------------------------------------------------------------------------------------------------------------------------------------------------------------------------------------------------------------------------------------------------------------------------------------------------------------------------------------------------

\subsection{Two-fold automorphisms and stability of graphs}

The following definition was first introduced by Zelinka in~\cite{Zelinka1971,Zelinka1972} for digraphs in his study of isotopies of digraphs and was extended to mixed graphs by Lauri, Mizzi and Scapellato in~\cite{LMS2015}. 

\begin{definition}\label{def:two-fold-auto}
Let $\Gamma$ be a graph. A pair of permutations $(\alpha,\beta)$ of $V(\Gamma)$ is called a \emph{two-fold automorphism} of $\Gamma$ if for all $u,v\in V(\Gamma)$, $\{u,v\} \in E(\Gamma)$ if and only if $\{u^\alpha, v^\beta\} \in E(\Gamma)$. A two-fold automorphism $(\alpha,\beta)$ is said to be \emph{nontrivial} if $\alpha\neq\beta$. 
\end{definition}

The following lemma is from~\cite[Proposition~4.2]{MSZ1989} (see also~\cite[Theorem~3.2]{LMS2015}). 
\begin{lemma}
\label{TF}
A graph is unstable if and only if it has a nontrivial two-fold automorphism. 
\end{lemma}

The next result is essentially known in~\cite{LMS2015} and~\cite{MSZ1989}. We give its proof for the completeness of the paper. 

\begin{lemma}\label{R-thick}
Let $(\alpha,\beta)$ be a two-fold automorphism of a graph $\Gamma$. Then the following statements hold: 
\begin{enumerate}[{\rm (a)}]
\item $(\beta,\alpha)$ and $(\alpha^{-1},\beta^{-1})$ are two-fold automorphisms of $\Gamma$. 
\item If $(\gamma,\delta)$ is a two-fold automorphism of $\Gamma$, then so is $(\alpha\gamma,\beta\delta)$.
\item If $(\alpha,\beta)$ is nontrivial with $\alpha=1$ or $\beta=1$, then $\Gamma$ is $R$-thick. 
\item If $(\alpha,\beta)$ is nontrivial with $\alpha\in\Aut(\Gamma)$ or $\beta\in\Aut(\Gamma)$, then $\Gamma$ is $R$-thick. 
%\item If $\alpha$ and $\beta$ have different orders, then $\Gamma$ is $R$-thick. ???
\end{enumerate}
\end{lemma}

\begin{proof}
Parts~(a) and~(b) hold clearly by the definition of two-fold automorphisms. 

Suppose that $(\alpha,\beta)$ is nontrivial with $\alpha=1$, that is, $(1,\beta)$ is a nontrivial two-fold automorphism of $\Gamma$. Then $\beta\neq1$, and for $u,v\in V(\Gamma)$ we have 
\[
u\in N_{\Gamma}(v)\Leftrightarrow \{u,v\}\in E(\Gamma)\Leftrightarrow\{u,v\}^{(1,\beta)}=\{u,v^\beta\}\in E(\Gamma) \Leftrightarrow u\in N_{\Gamma}(v^\beta). 
\]
This implies that vertices $v$ and $v^\beta$ have the same neighborhood for each $v\in V(\Gamma)$, and so $\Gamma$ is $R$-thick as $\beta\neq1$. Similarly, the condition that $(\alpha,\beta)$ is nontrivial with $\beta=1$ also implies that $\Gamma$ is $R$-thick. This completes the proof of part~(c). 

Next suppose that $(\alpha,\beta)$ is nontrivial with $\alpha\in\Aut(\Gamma)$. Then $(\alpha,\alpha)$ is a two-fold automorphism of $\Gamma$, and so is $(1,\beta\alpha^{-1})$ by parts~(a) and~(b).
Moreover, $\beta\alpha^{-1}\neq1$ as $\alpha\neq\beta$. Hence $(1,\beta\alpha^{-1})$ is a nontrivial two-fold automorphism of $\Gamma$. This implies that $\Gamma$ is $R$-thick by part~(c). Similarly, the condition that $(\alpha,\beta)$ is nontrivial with $\beta\in\Aut(\Gamma)$ also implies the $R$-thickness of $\Gamma$, which completes the proof of part~(d).  
\end{proof}

%Finally suppose that $\alpha$ and $\beta$ have orders $m$ and $n$, respectively, such that $m\neq n$. Without loss of generality, assume $m<n$. Then part~(b) implies that $(1,\beta^m)=(\alpha^m,\beta^m)$ is a nontrivial two-fold automorphism of $\Gamma$, and thus part~(c) asserts that $\Gamma$ is $R$-thick. This proves part~(e). 
%\end{proof}

%-------------------------------------------------------------------------------------------------------------------------------------------------------------------------------------------------------------------------------------------------------------------------------------------------------------------------------------------------------

\subsection{$\Sigma$-automorphisms and stability of graph pairs}

The next definition, first given in \cite[Definition~1.7]{QXZZ2021}, is a generalization of two-fold automorphisms.

\begin{definition}\label{def:sigma-auto}
Let $\Gamma$ and $\Sigma$ be graphs with $V(\Sigma)=\{1,\dots,m\}$, and let $\alpha_1,\dots,\alpha_m$ be permutations of $V(\Gamma)$. We say that the $m$-tuple $(\alpha_1,\dots,\alpha_m)$ is a \emph{$\Sigma$-automorphism} of $\Gamma$ if for all $u,v\in V(\Gamma)$ and $\{i, j\} \in E(\Sigma)$, $\{u, v\} \in E(\Gamma)$ if and only if $\{u^{\alpha_i}, v^{\alpha_j}\} \in E(\Gamma)$. Such a $\Sigma$-automorphism $(\alpha_1,\dots,\alpha_m)$ of $\Gamma$ is said to be \emph{non-diagonal} if there exists at least one pair of vertices $i,j\in V(\Sigma)$ such that $\alpha_i\neq\alpha_j$. 
\end{definition}

The following lemma is from~\cite[Lemma~2.6(a)]{QXZZ2021}. 
\begin{lemma}
\label{nontrivial}
Let $\Gamma$ and $\Sigma$ be graphs. If at least one $\Sigma$-automorphism of $\Gamma$ is non-diagonal, then $(\Gamma,\Sigma)$ is unstable.
\end{lemma}

The next lemma, which was proved in \cite[Theorem~1.8]{QXZZ2021}, characterizes nontrivially unstable graph pairs $(\Gamma,\Sigma)$ in the case when $\Gamma$ and $\Sigma$ are regular with coprime valencies and nontrivial automorphism groups and $\Sigma$ is vertex-transitive. 

\begin{lemma}
\label{Coprime}
Let $\Gamma$ and $\Sigma$ be regular graphs of coprime valencies with $\Sigma$ vertex-transitive. Suppose that both $\Gamma$ and $\Sigma$ are connected and $R$-thin and that at least one of them is non-bipartite. Then $(\Gamma,\Sigma)$ is nontrvially unstable if and only if at least one $\Sigma$-automorphism of $\Gamma$ is non-diagonal.
\end{lemma}

%==================================================================================================================================================================================================
\section{Proof of Theorem~\ref{thm1}}\label{sec-proof1}
%==================================================================================================================================================================================================

\begin{lemma}\label{thm0}
Let $\Gamma$ and $\Sigma$ be regular graphs of coprime valencies with $\Sigma$ vertex-transitive. If $(\Gamma, \Sigma)$ is nontrivially unstable, then $\Gamma$ is unstable. 
\end{lemma}

\begin{proof}
Let $V(\Sigma)=\{1,\dots,m\}$. Suppose that $(\Gamma, \Sigma)$ is nontrivially unstable. Then Lemma~\ref{Coprime} asserts that there exists a non-diagonal $\Sigma$-automorphism of $\Gamma$, say, $(\alpha_1, \ldots,  \alpha_m)$. Since $\Sigma$ is connected and the $\Sigma$-automorphism $(\alpha_1, \ldots,  \alpha_m)$ is non-diagonal, there exist adjacent vertices $i, j$ of $\Sigma$ such that $\alpha_i\neq\alpha_j$. Since $(\alpha_1, \ldots,  \alpha_m)$ is a $\Sigma$-automorphism of $\Gamma$, for $u,v\in V(\Gamma)$, $\{u, v\} \in E(\Gamma)$ if and only if $\{u^{\alpha_i}, v^{\alpha_j}\} \in E(\Gamma)$. Thus $(\alpha_i, \alpha_j)$ is a nontrivial two-fold automorphism of $\Gamma$. This together with Lemma~\ref{TF} implies that $\Gamma$ is unstable. 
\end{proof}

% The following corollary follows immediately from Lemma~\ref{thm0} and will play an important role in our proof of Theorem~\ref{thm1} and Theorem~\ref{thm2}. 

% \begin{corollary}\label{cor1}
% Let $\Gamma$ and $\Sigma$ be regular graphs of coprime valencies. Suppose that $\Gamma$ is non-bipartite and $\Sigma$ is vertex-transitive. If $(\Gamma, \Sigma)$ is nontrivially unstable, then $\Gamma$ is nontrivially unstable. 
% \end{corollary}

We are now in a position to prove part~(a) of Theorem~\ref{thm1}.

%\begin{theorem}\label{thm-tour}
%Let $(\Gamma,\Sigma)$ be a graph pair such that $\Gamma$ and $\Sigma$ are regular graphs of coprime valencies and $\Sigma$ is a vertex-transitive bipartite graph with a tour. Then $(\Gamma, \Sigma)$ is nontrivially unstable if and only if $\Gamma$ is nontrivially unstable. 
%\end{theorem}

\begin{proof}[Proof of part \emph{(a)}]
First suppose that $(\Gamma,\Sigma)$ is nontrivially unstable. As $\Sigma$ is bipartite, we obtain from Definition \ref{def-nontrivial-stable} that $\Gamma$ and $\Sigma$ are coprime connected $R$-thin graphs with $\Gamma$ non-bipartite. Since $\Sigma$ is vertex-transitive and the valencies of $\Gamma$ and $\Sigma$ are coprime, Lemma~\ref{thm0} implies that $\Gamma$ is unstable. Hence $\Gamma$ is nontrivially unstable. 

Next suppose that $\Gamma$ is nontrivially unstable. Then $\Gamma$ is connected, $R$-thin and non-bipartite. Since $\Gamma$ is unstable, we derive from Lemma~\ref{TF} that there exists a nontrivial two-fold automorphism $(\alpha,\beta)$ of $\Gamma$, that is, $\alpha$ and $\beta$ are distinct permutations of $V(\Gamma)$ such that 
\begin{equation}\label{eq3}
\{u,v\} \in E(\Gamma)\Leftrightarrow \{u^\alpha, v^\beta\} \in E(\Gamma)
\end{equation}
for $u,v\in V(\Gamma)$. Let $V(\Sigma)=\{1,\ldots,m\}$. Since $\Sigma$ is bipartite, $V(\Sigma)$ can be partitioned into two nonempty subsets $A$ and $B$ such that every edge of $\Sigma$ joins a vertex in $A$ and a vertex in $B$. For $i\in\{1,\dots,m\}$, let $\alpha_i=\alpha$ if $i\in A$ and $\alpha_i=\beta$ if $i\in B$.
Then it follows from~\eqref{eq3} that $(\alpha_1,\dots,\alpha_m)$ is a $\Sigma$-automorphism of $\Gamma$. Since $A$ and $B$ are nonempty, the $\Sigma$-automorphism $(\alpha_1,\dots,\alpha_m)$ is non-diagonal. Hence $(\Gamma, \Sigma)$ is nontrivially unstable by Lemma~\ref{Coprime}. 
\end{proof}
%-------------------------------------------------------------------------------------------------------------------------------------------------------------------------------------------------------------------------------------------------------------------------------------------------------------------------------------------------------

% \begin{corollary}\label{thm2}
% Let $\Gamma$ and $\Sigma$ be regular graphs of coprime valencies. Suppose that $\Sigma$ is a vertex-transitive bipartite graph of even valency. Then $(\Gamma, \Sigma)$ is nontrivially unstable if and only if $\Gamma$ is nontrivially unstable. 
% \end{corollary}

% \begin{proof}
% Since $\Sigma$ is of even valency, we know that $\Sigma$ is eulerian (see, for example,~\cite[Theorem~4.1]{BM1976}). Then there exists a Euler tour of $\Sigma$. Thus $\Sigma$ is a vertex-transitive bipartite graph with a tour. This together with Theorem~\ref{thm-tour} implies that 
% $(\Gamma, \Sigma)$ is nontrivially unstable if and only if $\Gamma$ is nontrivially unstable. 
% \end{proof}

To prove part~(b) of Theorem~\ref{thm1}, we need the next lemma.

\begin{lemma}\label{cvn}
Let $\Gamma$ and $\Sigma$ be graphs. 
\begin{enumerate}[{\rm (a)}]
\item $\Gamma \times \Sigma$ is non-bipartite if and only if both $\Gamma$ and $\Sigma$ are non-bipartite (\cite[Exercise 8.13]{HIK2011}). 
\item $\Gamma \times \Sigma$ is $R$-thin if and only if both $\Gamma$ and $\Sigma$ are $R$-thin (\cite[Lemma~2.3]{QXZ2019}).
\item Suppose that both $\Gamma$ and $\Sigma$ are connected with order at least $2$. Then $\Gamma \times \Sigma$ is connected if at least one of $\Gamma$ or $\Sigma$ is non-bipartite, and $\Gamma \times \Sigma$ has exactly two components if both $\Gamma$ and $\Sigma$ are bipartite (\cite{Weichsel62}, see also~\cite[Theorem~5.9]{HIK2011}).
\end{enumerate}
\end{lemma}

We are now ready to complete the proof of Theorem~\ref{thm1}. 

\begin{proof}[Proof of part \emph{(b)}]
If $\Sigma$ is disconnected or $R$-thick, then by Definition~\ref{def-nontrivial-stable}, the graph pair $(\Gamma,\Sigma)$ cannot be nontrivially unstable.
For the rest of the proof, suppose that $\Sigma$ is connected and $R$-thin. Then the assumption in part (b) of Theorem~\ref{thm1} implies that $\Sigma$ is non-bipartite. Suppose for a contradiction that $(\Gamma,\Sigma)$ is nontrivially unstable. Then $\Gamma$ is connected and $R$-thin by Definition \ref{def-nontrivial-stable}.

If $\Gamma$ is non-bipartite, then we deduce from Lemma~\ref{cvn} that $\Gamma\times\Sigma$ is connected, $R$-thin and non-bipartite, and hence Theorem~\ref{classic} gives $\Aut(\Gamma\times\Sigma)\cong\Aut(\Gamma)\times\Aut(\Sigma)$, contradicting the assumption that $(\Gamma,\Sigma)$ is unstable. Thus $\Gamma$ is bipartite. 

Let $V(\Sigma)=\{1,\ldots,m\}$. By Lemma~\ref{Coprime}, there exists a non-diagonal $\Sigma$-automorphism of $\Gamma$, denoted by $(\alpha_1, \ldots,  \alpha_m)$. 
% Since $\Sigma$ is connected and the $\Sigma$-automorphism $(\alpha_1, \ldots,  \alpha_m)$ is non-diagonal, there exist adjacent vertices $i,j\in V(\Sigma)$ such that $\alpha_i\neq\alpha_j$. 
Since $\Sigma$ is non-bipartite and vertex-transitive, we derive that each vertex of $\Sigma$ lies in some odd cycle in $\Sigma$. 

Consider an arbitrary odd cycle $C=(i_1,\dots,i_s)$ in $\Sigma$, where $s\geq3$ is odd. Write $\beta_r=\alpha_{i_r}$ for $r\in\{1,\dots,s\}$ and count the subscripts of $\beta_r$ modulo $s$. According to the definition of a $\Sigma$-automorphism (see Definition \ref{def:sigma-auto}), for $u,v\in V(\Gamma)$, we have
\[
\{u, v\} \in E(\Gamma)\Leftrightarrow\{u^{\beta_r}, v^{\beta_{r+1}}\} \in E(\Gamma)\Leftrightarrow\{u^{\beta_{r+1}}, v^{\beta_{r+2}}\} \in E(\Gamma).
\]
In other words, both $(\beta_r,\beta_{r+1})$ and $(\beta_{r+1},\beta_{r+2})$ are two-fold automorphisms of $\Gamma$. Then by parts (a) and (b) of Lemma~\ref{R-thick} we conclude that $(1,\beta_{r+2}\beta_r^{-1})$ is a two-fold automorphisms of $\Gamma$. Since $\Gamma$ is $R$-thin, part (c) of Lemma~\ref{R-thick} implies that $\beta_{r+2}\beta_r^{-1}=1$, that is, $\beta_{r+2}=\beta_r$. Since this holds for every $r\in\{1,\dots,s\}$ and $s$ is odd, it follows that $\beta_1=\dots=\beta_s$.
Consequently, we may assign a permutation $\alpha_C$ to the odd cycle $C$ such that $\alpha_C=\alpha_{i_1}=\dots=\alpha_{i_s}$. 

Define two odd cycles $C$ and $D$ in $\Sigma$ to be equivalent if $\alpha_C=\alpha_D$.
Under this equivalence relation, the set of odd cycles in $\Sigma$ are partitioned into equivalence classes $\calC_1,\dots,\calC_t$. 
For $i\in\{1,\dots,t\}$, let $V_i$ be the union of the vertex sets of the cycles in $\calC_i$ and let $\sigma_i$ be the assigned permutation to any of the cycles in $\calC_i$. Then $\alpha_k=\sigma_i$ for all $k\in V_i$, and $\sigma_i\neq\sigma_j$ for $i\neq j$. 
As a consequence, $V_i\cap V_j=\emptyset$ for $i\neq j$. Since every vertex of $\Sigma$ lies in some odd cycle in $\Sigma$, the set $V(\Sigma)$ is covered by $\{V_1,\dots,V_t\}$, and so $\{V_1,\dots,V_t\}$ is a partition of $V(\Sigma)$. Since $(\alpha_1,\ldots,\alpha_m)$ is non-diagonal, we have $t\geq2$.
Now the connectivity of $\Sigma$ implies that there exist distinct $i,j\in\{1,\dots,t\}$ and vertices $x\in V_i$ and $y\in V_j$ with $\{x,y\}\in E(\Sigma)$.
By Definition \ref{def:sigma-auto}, for $u,v\in V(\Gamma)$, we have
\[
\{u, v\} \in E(\Gamma)\Leftrightarrow\{u^{\alpha_x}, v^{\alpha_y}\} \in E(\Gamma).
\]
As $\alpha_x=\sigma_i$, $\alpha_y=\sigma_j$ and $\sigma_i\neq\sigma_j$, it follows that $(\sigma_i,\sigma_j)$ is a nontrivial two-fold automorphism of $\Gamma$.
Let $\{z,w\}$ be an edge of any cycle in $\calC_i$. Again, by Definition \ref{def:sigma-auto}, for $u,v\in V(\Gamma)$, we have
\[
\{u, v\} \in E(\Gamma)\Leftrightarrow\{u^{\alpha_z}, v^{\alpha_w}\} \in E(\Gamma).
\]
Since $\alpha_z=\alpha_w=\sigma_i$, this means that $\sigma_i\in\Aut(\Gamma)$. So part (d) of Lemma~\ref{R-thick} yields that $\Gamma$ is $R$-thick, a contradiction. 
\end{proof}

%==================================================================================================================================================================================================
\section{Applications and concluding remarks}\label{sec-application}
%==================================================================================================================================================================================================

\begin{proposition}\label{K_m}
Let $m\geq3$ be an integer, and let $\Gamma$ be a graph of valency coprime to $m-1$ such that $\Aut(\Gamma)\neq1$. Then the following statements hold:
\begin{enumerate}[{\rm (a)}]
\item If $\Gamma$ is connected and $R$-thin, then $(\Gamma,K_m)$ is stable.
\item If $\Gamma$ is disconnected or $R$-thick, then $(\Gamma,K_m)$ is trivially unstable. 
\end{enumerate}
\end{proposition}

\begin{proof}
Since the complete graph $K_m$ is vertex-transitive and non-bipartite with valency $m-1$, we conclude from part (b) of Theorem~\ref{thm1} that the graph pair $(\Gamma,K_m)$ cannot be nontrivially unstable. In view of Definition~\ref{def-nontrivial-stable}, this implies that if $\Gamma$ is connected and $R$-thin, then $(\Gamma,K_m)$ is stable. On the other hand, if $\Gamma$ is disconnected or $R$-thick, then $(\Gamma,K_m)$ is trivially unstable by Theorem~\ref{TriviallyStable} and Definition~\ref{def-nontrivial-stable}.
\end{proof}

Proposition~\ref{K_m} is not true if the valency of $\Gamma$ is not coprime to $m-1$. For example, let $\Gamma=K_m\times K_2$. Then $\Gamma$ is a connected $R$-thin graph of valency $m-1$. However, $(\Gamma,K_m)$ is trivially unstable as $\Gamma$ and $K_m$ are not coprime. 
Note also that we do need $m\geq3$ in Proposition~\ref{K_m}, as any nontrivially unstable graph $\Gamma$ with $\Aut(\Gamma)\neq1$ would satisfy the assumption but not the conclusion of part~(b) if $m=2$ in Proposition~\ref{K_m}.
% \zhou{It seems that this example does not support the first sentence in this paragraph as the valency $m$ of $\Gamma$ is coprime to $m-1$. Perhaps I missed some point here.}

% \zhou{new paragraph?}  
Since $K_m$ is non-bipartite when $m \ge 3$, the condition that $\Gamma$ is of valency coprime to $m-1$ in Proposition~\ref{K_m} implies that $\Gamma$ and $K_m$ are coprime. So it is natural to ask whether the results in Proposition~\ref{K_m} are still true if this condition is replaced by the weaker condition that $\Gamma$ and $K_m$ are coprime. In fact, by Theorem~\ref{TriviallyStable} and Definition~\ref{def-nontrivial-stable}, we know that part (b) of Proposition~\ref{K_m} is true if $\Gamma$ and $K_m$ ($m \geq 3$) are coprime and $\Aut(\Gamma)\neq1$. We conjecture that part (a) of Proposition~\ref{K_m} is also true under the same condition.

\begin{conjecture}\label{conjecture}
Let $m\geq 3$ be an integer, and let $\Gamma$ be a graph coprime to $K_m$ such that $\Aut(\Gamma)\neq1$. If $\Gamma$ is connected and $R$-thin, then $(\Gamma,K_m)$ is stable.
\end{conjecture}

All connected arc-transitive graphs of order $2$ to $47$ are known (see the list \url{https://www.math.auckland.ac.nz/~conder/symmetricgraphs-orderupto47-byedges.txt} constructed by M. Conder). Computing in Magma shows that Conjecture \ref{conjecture} is true when $\Gamma$ is one of these graphs and $3 \le m \le 25$, or $\Gamma$ is a connected arc-transitive graph with order at most $31$ and $3 \le m \le 100$. Our computing also shows that Conjecture \ref{conjecture} is true when $\Gamma$ is any connected bipartite graph of order at most $47$ in the database \url{https://hog.grinvin.org} and $3 \le m \le 100$. It is not difficult to see that this conjecture is true if $\Gamma$ is non-bipartite and $m\geq 3$.

\begin{proposition}\label{C_m}
Let $m\geq3$ be an integer, and let $\Gamma$ be a connected $R$-thin graph of odd valency such that $\Aut(\Gamma)\neq1$. Then the following statements hold:
\begin{enumerate}[{\rm (a)}]
\item If $m$ is odd, then $(\Gamma,C_m)$ is stable.
\item If $m=4$, then $(\Gamma,C_m)$ is trivially unstable. 
\item If $m\geq6$ is even, then $(\Gamma,C_m)$ is stable or nontrivially unstable or trivially unstable, respectively, if and only if $\Gamma$ is stable or nontrivially unstable or trivially unstable. 
\end{enumerate}
\end{proposition}

\begin{proof}
First suppose that $m$ is odd. In this case, $C_m$ is a connected $R$-thin vertex-transitive non-bipartite graph of valency $2$. Since the valency of $\Gamma$ is odd,  $(\Gamma,C_m)$ cannot be nontrivially unstable by part (b) of Theorem~\ref{thm1}. Since $\Gamma$ is connected and $R$-thin, it follows from Definition~\ref{def-nontrivial-stable} that $(\Gamma,C_m)$ is stable, as part~(a) asserts.  

Next suppose that $m=4$. Then $C_m$ is $R$-thick, and so $(\Gamma,\Sigma)$ is trivially unstable by Theorem~\ref{TriviallyStable} and Definition~\ref{def-nontrivial-stable}. This proves part~(b).

Finally, suppose that $m\geq6$ is even. Then $C_m$ is a connected $R$-thin vertex-transitive bipartite graph of valency $2$. Since $\Gamma$ has odd valency, part (a) of Theorem~\ref{thm1} shows that $(\Gamma,C_m)$ is nontrivially unstable if and only if $\Gamma$ is nontrivially unstable. Moreover, it follows from Theorem~\ref{TriviallyStable} and Definition~\ref{def-nontrivial-stable} that $(\Gamma,C_m)$ is trivially unstable if and only if $\Gamma$ is disconnected or $R$-thick or bipartite, which happens exactly when $\Gamma$ is trivially unstable. This completes the proof of part~(c).
\end{proof}

% {\color{red} To do: mention what if the valency of $\Gamma$ is even. Either establish more results or pose a conjecture/question.}

Note that the condition that $\Gamma$ is of odd valency in Proposition~\ref{C_m} ensures that $\Gamma$ and $C_m$ are coprime. It is natural to ask whether Proposition~\ref{C_m} is still true for graphs $\Gamma$ of even valency under the additional condition that $\Gamma$ and $C_m$ are coprime. In particular, we pose the following question. 

\begin{question}
For a stable graph $\Gamma$ and an even integer $m\geq 6$, under what condition is $(\Gamma,C_m)$ nontrivially unstable?
\end{question}

%{\color{red} To do: add some concluding remarks. For example, what if we delete the condition that $\Sigma$ is vertex-transitive in our results?}

Theorem~\ref{thm1} relates the nontrivial instability of $(\Gamma,\Sigma)$ to that of $\Gamma$ in the case when $\Gamma$ and $\Sigma$ are of coprime valencies and $\Sigma$ is vertex-transitive. It would be interesting to study when a similar relation exists if the valencies of $\Gamma$ and $\Sigma$ are not coprime or $\Sigma$ is not vertex-transitive. In general, we pose the following question.

\begin{question}
Let $\Gamma$ and $\Sigma$ be regular graphs. Under what condition does it hold that $(\Gamma, \Sigma)$ is nontrivially unstable if and only if $\Gamma$ is nontrivially unstable?
\end{question}

\vskip0.1in
\noindent\textsc{Acknowledgements.}  
 The first author was supported by the National Natural Science Foundation of China (12101421).

%We are grateful to the anonymous referees for their helpful comments and Dr Jiyong Chen for his valuable advices. Part of the work was done during a visit of the first author to The University of Melbourne. The first author would like to thank The University of Melbourne for its hospitality during her visit and Beijing Jiaotong University for its financial support. She also thanks the National Natural Science Foundation of China (11671030) for its financial support during her PhD. The first author was supported by the Fundamental Research Funds for Beijing Universities allocated to Capital University of Economics and Business(XRZ2020058). The third author was supported by the National Natural Science Foundation of China (11671030,12071023).

\end{document}